\documentclass{amsart}
\usepackage{tikz}
\usetikzlibrary{fadings,arrows}

\usepackage{amsmath, amssymb, mathrsfs, verbatim, multirow}
\usepackage[all]{xy}
\usepackage{pifont}
\usepackage{float}

\newtheorem{Teo}{Theorem}[section]
\newtheorem{Prop}[Teo]{Proposition}
\newtheorem{Lema}[Teo]{Lemma}
\newtheorem{Cor}[Teo]{Corollary}

\theoremstyle{definition}
\newtheorem{Def}[Teo]{Definition}

\newtheorem{Obs}[Teo]{Remark}

\newtheorem{Exa}[Teo]{Example}

\newcommand{\R}{\mathbb{R}}
\newcommand{\Z}{\mathbb{Z}}
\newcommand{\N}{\mathbb{N}}

\newcommand{\lra}{\longrightarrow}

\newcommand{\VR}{\mathcal{O}}

\newcommand{\SU}{\mbox{\rm supp}}

\newcommand{\ann}{\mbox{\rm ann}}

\newcommand{\IN}{\mbox{\rm in}}

\begin{document}
\title[K\"ahler differentials]{The module of K\"ahler differentials for extensions of valuation rings}

\author{Josnei Novacoski}
\address{Departamento de Matem\'{a}tica,  Instituto de Ciências Matemáticas e de Computação, Avenida Trabalhador São-carlense 400, 13565--905, S\~ao Carlos -SP, Brazil}
\email{josnei@icmc.usp.br}
\author{Mark Spivakovsky}
\address{CNRS UMR 5219, Institut de Mathématiques de Toulouse, 118, rte de Narbonne, 31062 Toulouse cedex 9, France and 
Instituto de Matem\'aticas (Cuernavaca) LaSol, UMI CNRS 2001, UNAM, Av. Universidad s/n. Col. Lomas de Chamilpa, 62210, Cuernavaca, Morelos, M\'exico}
\email{mark.spivakovsky@math.univ-toulouse.fr}

\thanks{During the realization of this project the authors were supported by two grants from Funda\c{c}\~ao de Amparo \`a Pesquisa do Estado de S\~ao Paulo (process numbers 2017/17835-9 and 2019/21181-0).}

\begin{abstract}
The main goal of this paper is to describe the module of K\"ahler differentials for an extension of valuation rings. More precisely, we consider a simple algebraic valued field extension $(L/K,v)$ and the corresponding valuation ring extension $\VR_K\hookrightarrow\VR_L$. In the case when $e(L/K,v)=1$ (and under two additional conditions) we present a characterization of $\Omega_{\VR_L/\VR_K}$ in terms of a given sequence of key polynomials for the extension. We use our main result to present criteria for the equality $\Omega_{\VR_L/\VR_K}=(0)$ to hold.
\end{abstract}

\maketitle
\section{Introduction}

Let $(L/K,v)$ be a simple algebraic extension of valued fields and denote by $\VR_L$ and $\VR_K$ the corresponding valuation rings. We denote by $vL$ and $vK$ the value groups of $(L,v)$ and $(K,v)$, respectively. The main purpose of this paper is to present, under certain assumptions, a characterization of the module of K\"ahler differentials $\Omega:=\Omega_{\VR_L/\VR_K}$ of the ring extension $\VR_K\hookrightarrow \VR_L$.

In order to describe $\Omega$, the first step is to present a set of generators for this extension. Throughout this paper, we will assume that
\begin{description}
\item[(A1)] $e(L/K,v):=(vL:vK)=1$.
\end{description}
Under assumption \textbf{(A1)}, \emph{complete sets} provide a natural tool to obtain such sets of generators (for more details see Section \ref{completesets}). However, one also needs to have an explicit description of the relations between these generators. In \cite{NS20251}, we presented a complete set whose relations are as simple as possible. For  the sake of exposition, we outline the description of $\VR_L$ in terms of generators and relations over $\VR_K$. A more detailed description will be presented in Section \ref{genneratosrorelasi}.

For a fixed generator $\eta$ of $L/K$, the construction in \cite{NS20251} provides explicit sets of polynomials $\textbf{Q}=\{ Q_i\}_{i\in I}\subseteq K[x]$ and $\tilde{\textbf{Q}}=\{\tilde Q_i\}_{i\in I}\subseteq \VR_K[x]$ such that the set $\tilde{\textbf{Q}}(\eta)=\{\tilde Q_i(\eta)\}_{i\in I}$ is a set of generators of $\VR_L$ over $\VR_K$. The index set $I$ admits a total ordering, under which it has a maximum $i_{\text{max}}$. For each $i\in I^*:=I\setminus\{i_{\rm max}\}$ choose an independent variable $X_i$ and let $\textbf X=\{X_i\}_{i\in I^*}$. By \cite[Proposition 2.1]{NS20251}, the natural map
$\pi:\VR_K[{\bf X}]\rightarrow L$ defined by $X_i\mapsto \tilde Q_i(\eta)$ induces an isomorphism
\[
\VR_L\simeq \VR_K[\textbf{X}]/\mathcal I,
\]
where $\mathcal I$ is the kernel of $\pi$. In \cite{NS20251}  we presented an explicit description of the set of generators of $\mathcal I$.

The set $I$ can be written as a finite disjoint union
\begin{equation}\label{decompI}
I=\coprod\limits_{q=1}^{{w+1}}I^{(q)},
\end{equation}
where the $I^{(q)}$ are segments in $I$ with $I^{(1)}<I^{(2)}<\dots<I^{(w)}<I^{(w+1)}=\{i_{\text{max}}\}$ such that for each $q$, $1\leq q\leq w+1$, all the polynomials $Q_i$ for $i\in I^{(q)}$ have the same degree $d_q$. Moreover, for $q<q'$ we have $d_{q}<d_{q'}$.  We will use the notation
\[
\textbf Q_{d_q}:=\left\{Q_i\ \left|\ i\in I^{(q)}\right.\right\}\mbox{ and }\textbf{Q}_n:=\emptyset\mbox{ if }n\notin \{d_1,\ldots,d_{w+1}\}.
\]

The second assumption we will make throughout this paper is that for every $q$, $1\leq q\leq w+1$, we have
\begin{description}
\item[(A2)] Either $\#I^{(q)}=1$ or $I^{(q)}$ is order-isomorphic to $\N$.
\end{description}
Here $\N$ denotes the set of the non-negative integers with the usual order. For $q\in\{1,\dots,w,w+1\}$, let $\ell_q:=\min\ I^{(q)}$, so that $I^{(q)}$ is equal to either $\left\{\ell_q\right\}$ or to $\left\{\left.\ell_q+s\ \right|\ s\in\N\right\}$. For $q\in\{1,\dots,w,w+1\}$ and $i,i'\in I$ with $i'\in I^{(q)}$, we will write $i\ll i'$ if $i<\ell_q$.

The third main assumption,
\begin{description}
\item[(A3)] $\textbf Q$ is successively strongly monic,
\end{description}
ensures that for every pair of indices $(\ell,i)$ such that $\tilde Q_i\in I^{(q)}$ and $\tilde Q_\ell\in I^{(q+1)}$, the polynomial $\tilde Q_\ell$ can be written in a suitable way in terms of $\tilde Q_i$ (see Definition \ref{stronglymonis}). The assumption that $\bf Q$ is successively strongly monic is not necessary, but makes the computations easier. This condition is always satisfied if, for instance, $(K,v)$ is \emph{henselian}, or the \emph{rank} of $v$ is one, as well as in the case of \emph{pure extensions} $(L/K,v)$. The general case will be treated in future work.

The crucial part of the description of $\Omega$ are the notions of \emph{neat pairs $(i,\ell)$} and \emph{neat expansions} (see Section \ref{preseent}). In \cite{NS20251} we construct, for every neat pair $(i,\ell)$, a neat polynomial $Q_{\ell i}\in \VR_K[\textbf X]$ and an element $b_{\ell i}$ and define the ideals
\begin{equation}
\mathcal I_1=\left(\left.b_{\ell i}X_\ell-Q_{\ell i}\ \right|\  i,\ell\in I^*, (\ell,i)\text{ is neat}\right)\VR_K[\bf X]\label{eq:defI1}
\end{equation}
and
\begin{equation}
\mathcal I_2=\left(Q_{i_{\max}i}\ \left|\ i\in I^{(w)}\right.\right)\VR_K[\bf X].\label{eq:defI2}
\end{equation}
The main result of \cite{NS20251} is that $\mathcal I=\mathcal I_1+\mathcal I_2$.
 
From this description of $\mathcal I$ we can deduce our first main result:
\begin{Teo}\label{mainthm}
Assume that the conditions \textbf{(A1)}, \textbf{(A2)} and \textbf{(A3)} above are satisfied. Then
\begin{equation}
\Omega=\frac{\bigoplus\limits_{i\in I^*}\VR_LdX_i}{\mathcal J}\mbox{ where }\mathcal J=\mathcal J_1+\mathcal J_2\label{eq:Omegaasaquotient}
\end{equation}
with
\begin{equation}
\mathcal J_1=\left(\left.b_{\ell i}dX_\ell-dQ_{\ell i}\ \right|\ i,\ell\in I^*,\ (\ell,i)\text{ is neat}\right)\mbox{ and }\mathcal J_2=\left(\left.dQ_{i_{\text{max}}i}\ \right|\ i\in I^{(w)}\right).\label{eq:generatorsCalI1}
\end{equation}
\end{Teo}
\begin{Obs} A set of generators of $\Omega$ as an $\VR_L$-module is given by \cite[Proposition 2.1]{NS20251}: we have $\Omega=\sum\limits_{\tilde Q\in\tilde{\textbf{Q}}}\VR_Ld\left(\tilde Q(\eta)\right)$. The main content of Theorem \ref{mainthm} is describing the relations between these generators.
\end{Obs}

We use Theorem \ref{mainthm} to give necessary and sufficient conditions for the equality
\begin{equation}
\Omega=(0)\label{eq:Omega=(0)}
\end{equation}
to hold.

An important condition in this context is called {\bf the invertibility hypothesis} (Definition \ref{invertibilityhypothesis}). First, we define {\bf the invertibility hypothesis at level} $q$ for each $q\in\{1,\dots,w\}$. If $I^{(q)}$ is a singleton, this hypothesis is, roughly speaking, a form of ``separability" of the equation expressing the key polynomial $Q_{\ell_{q+1}}$ in terms of the preceding key polynomial $Q_{\ell_q}$. If $I^{(q)}$ has order type $\N$, then $Q_{\ell_{q+1}}$ is a \emph{limit key polynomial} for ${\bf Q}_{d_q}$. In this case, the invertibility hypothesis guarantees that for each $i\in I^{(q)}$ there exists $j\in I^{(q+1)}$ such that the element $d\left(\tilde Q_i(\eta)\right)$ is an $\VR_L$-multiple of $d\left(\tilde Q_j(\eta)\right)$ in $\frac\Omega{\sum\limits_{\ell<\ell_q}d\left(\tilde Q_\ell(\eta)\right)}$.

We say that the invertibility hypothesis holds for the valued field extension $(L/K,v)$ if it holds at level $q$ for all $q\in\{1,\dots,w\}$. Theorem \ref{MainInvertibility1} says that the invertibility hypothesis implies the equality \eqref{eq:Omega=(0)}. The theorem also gives a partial converse to this statement: the equality \eqref{eq:Omega=(0)} implies that the invertibility hypothesis holds at level $w$.

The results of this paper can be viewed as a vast generalization of those results in \cite{NovSpiAnn}  that pertain to unramified extensions. 

One motivation for our study is to better understand the defect of an extension of valued fields. In \cite{Kuhl}, F.-V. Kuhlmann introduces a classification of defect extensions of degree $p$ as dependent and independent. In \cite{Nov21} it is proved that this classification can be read off from the sequence of key polynomials that define this extension. More recently, Cutkosky, Kuhlmann and Rzepka (see \cite{CKR}) proved that one can use the module of K\"ahler differentials to characterize when a defect extension of degree $p$ is dependent or independent. In \cite{CK} the authors present criteria for the equality \eqref{eq:Omega=(0)} to be satisfied.

Among other things, our Proposition \ref{notDeltabounded1} exhibits (roughly speaking) the relation between the invertibility hypothesis and the extension $(L/K,v)$ having independent defect.

This paper is organized as follows. In Section \ref{completesets} we recall the concept of a complete set and explain how it induces a set of generators of the extension
$\VR_L/\VR_K$. In Section \ref{preseent} we introduce $i$-th expansions and in Section \ref{Immediatesuccessors} we describe the successors of a given element of
$\textbf{Q}$. They will be important in order to present the relations between the generators of $\VR_L/\VR_K$.

Section \ref{Characteri1} is crucial for the description of $\Omega$. In this section we survey the results of \cite{NS20251} whose main result is Proposition \ref{generatorsCalI}. This proposition says, roughly, that all the relations involving key polynomials are generated by the relations between successive key polynomials and those between $Q_{i_{\rm max}}$ (that is, the minimal polynomial of $\eta$ over $K$) and the other elements of $\textbf{Q}$. The main ingredient of the proof of Proposition \ref{generatorsCalI} is the concept of \textit{neat polynomials}. This concept gives a certain rigidity to the relations between the generators.

In Section \ref{Kahler} we use Proposition \ref{generatorsCalI} to present an explicit description of $\Omega$. We also express $\Omega$ as a limit of naturally chosen submodules.

In order to present an estimation for the annihilator of $\Omega$ we introduce the concept of \textit{multilines of width $w$}, where $w\in\N$. The module $\Omega$ has a natural structure of multiline. The results of Section \ref{multilies} allow us to present an explicit description of the annihilator of a multiline (Remark \ref{Descriptionannihilator}). When applied to $\Omega$, this gives an estimation for the annihilator of $\Omega$. This is a vast generalization of \cite[Proposition 4.10]{NovSpiAnn}, which deals almost exclusively with pure extensions $(L/K,v)$. Finally, we introduce the invertibility hypothesis. Roughly speaking, for each $q$ for which $I^{(q)}$ satisfies this hypothesis, we can guarantee that the upper bound for the width of $\Omega$ is reduced by one. In particular, if the invertibility hypothesis holds for every $q\in \{1,\ldots, w\}$, then we guarantee that the width of $\Omega$ is zero, that is, $\Omega=(0)$. We also present in Section \ref{multilies} partial converses to the latter result.

\section{Complete sets, full expansions and neat polynomials}\label{genneratosrorelasi}
\subsection{Complete sets}\label{completesets}
Let $L=K(\eta)$ and denote by $g$ the minimal polynomial of $\eta$ over $K$.  We consider the valuation $\nu$ on $K[x]$ (with support ${\rm supp}(\nu)=\{f\in K[x]\mid \nu f=\infty\}$ equal to $(g)$), defined by $v$ and $\eta$. This valuation is defined by
\[
\nu(f):=v(f(\eta)).
\]

Let $q,f\in K[x]$ be two polynomials, with $\deg(q)\geq 1$. There exist (uniquely determined) $f_0,\ldots,f_n\in K[x]$ with
$\deg(f_i)<\deg(q)$ for every $i$, $0\leq i\leq n$, such that
\[
f=f_0+f_1q+\ldots+f_nq^n.
\]
This expression is called the \textbf{$q$-expansion of $f$}. If $f_n=1$ in the above expansion, then we say that $f$ is \textbf{$q$-monic}. For a monic polynomial $q$ of $\deg(q)\geq 1$, the \textbf{truncation of $\nu$ at $q$} is defined as
\[
\nu_q(f)=\min\left\{\nu\left(f_iq^i\right)\right\},
\]
where $f=f_0+f_1q+\ldots+f_nq^n$ is the $q$-expansion of $f$.

For a set $\textbf{Q}\subseteq K[x]$ we denote by $\N^\textbf{Q}$ the set of mappings $\lambda:\textbf{Q}\lra \N$ such that $\lambda(q)=0$ for all but finitely many $q\in\textbf{Q}$. For $\lambda\in \N^\textbf{Q}$ we denote
\[
\textbf{Q}^\lambda:=\prod_{q\in\textbf{Q}}q^{\lambda(q)}\in K[x].
\]
For two subsets $\textbf{Q}$ and $\tilde{\textbf{Q}}$ of $K[x]$ we will say that they are $K${\bf-proportional} if there exists a bijection $\phi:\textbf{Q}\lra\tilde{\textbf{Q}}$ such that for every $q\in \textbf Q$ we have $\phi(q)=aq$ for some $a\in K$ (the element $a$ may depend on $q$).

A set $\textbf{Q}\subseteq K[x]\setminus K$ is called a {\bf complete} set for $\nu$ if for every $f\in K[x]$ there exists
$q\in\textbf{Q}$ such that
\begin{equation}\label{eqabotdegpol}
\deg(q)\leq\deg(f)\mbox{ and }\nu(f)=\nu_q(f).
\end{equation}

\begin{Prop}\label{GenerationofOLoverOK}\cite[Proposition 3.5]{NovSpiAnn}
Let $L=K(\eta)$ and take a valuation $v$ on $L$ such that $e(L/K,v)=1$. Consider the valuation $\nu$ on $K[x]$ defined by $v$ and $\eta$ and fix a complete set $\textbf{Q}$ for $\nu$. Then there exists a complete set $\tilde{\textbf{Q}}$ for $\nu$, $K$-proportional to $\textbf{Q}$, such that
\[
\VR_L=\VR_K\left[\tilde Q(\eta)\ \left|\ \tilde Q\in\tilde{\textbf{Q}}\right.\right].
\]
\end{Prop}

Let $\overline K$ denote the algebraic closure of $K$. Fix an extension $\overline \nu$ of $\nu$ to $\overline K[x]$. For each $f\in K[x]\setminus K$ we define
\[
\epsilon(f):=\max\{\overline \nu(x-a)\mid a\mbox{ is a root of }f\}.
\]
We also define $\epsilon(a)=-\infty$ if  $a\in K$. 
\begin{Def}
A monic polynomial $Q\in K[x]\setminus K$ is called \textbf{a key polynomial for} $\nu$ if $\epsilon(f)<\epsilon(Q)$ for every $f\in K[x]\setminus K$ with $\deg(f)<\deg(Q)$.
\end{Def}
Take a complete sequence of key polynomials $\textbf{Q}=\{Q_i\}_{i\in I}$ for $\nu$. This means that every element $Q_i$ is a key polynomial for $\nu$, the set $I$ is well ordered, the map $i\mapsto \epsilon(Q_i)$ is strictly increasing and the set $\{Q_i\}_{i\in I}$ is a complete set for $\nu$. In particular, $g$ is the last element of $\textbf{Q}$ and
\begin{equation}
\forall Q_i\in{\bf Q}\setminus\{g\}\mbox{ we have }\nu(Q_i)\in vL.\label{eq:nuQi<infty}
\end{equation}
The existence of such sequences of key polynomials is proved, for example, in \cite{Novspivkeypol}. For each $i\in I$, denote by $\nu_i$ the truncation of $\nu$ at $Q_i$.

For a subset $\gamma\subset vL\cup\{\infty\}$ let $\alpha(\gamma)$ denote the smallest final segment of
$vL\cup\{\infty\}$ containing $\gamma$; let
\[
I_\gamma:=\{y\in L\ |\ v(y)\ge a\text{ for some }a\in\gamma\}.
\]
The set $I_\gamma$ is a fractional ideal in $L$ (this is the same as saying that it is an $\mathcal O_L$-submodule of $L$).

\subsection{Full expansions and neat polynomials}\label{preseent}
Let $(K,v)$ be a valued field and take a valuation $\mu$ on $K[x]$, extending $v$. Here $\mu$ is allowed to have non-trivial support. For simplicity of notation, in this section we denote $\Gamma_v=vK$ and $\Gamma_\mu$ the group generated by $\{\mu(f)\mid f\in K[x]\setminus\SU(\mu)\}$.

\begin{Def} Given two valuations $\mu,\mu'$ of $K[x]$ extending $v|_K$, with $\Gamma_{\mu}$ and $\Gamma_{\mu'}$ contained in a common ordered abelian group
$\Lambda$, we say that $\mu\le\mu'$ if
\begin{equation}
\mu(f)\le\mu'(f)\label{eq:m'lemu}
\end{equation}
for all $f\in K[x]$. We say that $\mu<\mu'$ if, in addition, the inequality \eqref{eq:m'lemu} is strict for some $f\in K[x]$.
\end{Def}

\begin{Def}
For an increasing family $\mathcal F=\{\mu_j\}_{j\in J}$  of valuations on $K[x]$, without largest element, we say that $f\in K[x]$ is \textbf{$\mathcal F$-stable} if there exists $j\in J$ such that $\mu_j(f)=\mu_k(f)$ for every $k\in J$, $k\geq j$. Otherwise, it is called $\mathcal F$-unstable. A monic polynomial $\phi$ is called a \textbf{limit key polynomial} for $\mathcal F$ if it is $\mathcal F$-unstable and has smallest degree among $\mathcal F$-unstable polynomials.
\end{Def}

Let $\textbf{Q}=\{Q_i\}_{i\in I}$ be a sequence of key polynomials for $\nu$. Decompose $I$ as in \eqref{decompI}. It is easy to see that if $q\in\{1,\dots,w\}$ and $I^{(q)}$ is not a singleton, then $\{\nu_i\}_{i\in I^{(q)}}$ is an increasing family of valuations and every $Q_j$ such that $j\in I^{(q+1)}$ is a limit key polynomial for $\{\nu_i\}_{i\in I^{(q)}}$. 

The following fact is well known and easy to prove (for instance, see \cite[Corollary 3.13]{Nov11}). Given two indices $i,i'\in I^*$ with $i<i'$, we have
\begin{equation}
\nu_i<\nu_{i'}<\nu.\label{eq:nui<nui'}
\end{equation}
For each $i\in I$, since $e(L/K,v)=1$, there exists $a_i\in K$ such that $\nu(Q_i)=v(a_i)$. We set $\tilde Q_i:=\frac{Q_i}{a_i}$.

Consider a polynomial $f\in K[x]\setminus\{0\}$. We will now define a full $i$-th expansion of $f$ as an expression of the form
\begin{equation}\label{expiadicf}
f=\sum_{j=1}^rb_j{\bf\tilde Q}^{\lambda_j}\mbox{ with $b_j\in K, \lambda_j\in \textbf{Q}^\N$ and }\lambda_j\left(\tilde Q_k\right)=0\mbox{ if }k>i,
\end{equation}
having additional properties spelled out in the definitions below.

\begin{Def} For an index $k\in I^*$, we say that $\tilde Q_k$ {\bf appears in \eqref{expiadicf}} if
\[
\lambda_j\left(\tilde Q_k\right)>0
\]
for some $j\in\{1,\dots,r\}$.
\end{Def}

\noindent{\bf Notation.} For $n\in \N$ such that $\textbf{Q}_n\ne\emptyset$, put
\[
n_+=\min\left\{n'\in\N\ \left|\ n'>n,\textbf{Q}_{n'}\ne\emptyset\right.\right\}.
\]
\begin{Def} We define a \textbf{full $i$-th expansion of $f$} as an expression \eqref{expiadicf} having the following properties:
\begin{enumerate}
\item\begin{equation}
\nu_i(f)=\min_{1\leq j\leq r}\{v(b_j)\},\label{eq:nui(f)=minvbj}
\end{equation}
\item $\lambda_j\left(\tilde Q_k\right)<\left(\deg_xQ_k\right)_+$ whenever $k<i$ and $j\in\{1,\dots,r\}$
\item for each $n\in\N_{>0}$ we have
\[
\#\{k\in I^*\ |\ \deg_xQ_k=n\mbox{ and }\tilde Q_k\mbox{ appears in \eqref{expiadicf}}\}\le1.
\]   
\end{enumerate}
\end{Def}
To show that an $i$-th expansion of $f$ exists, let
\[
f=f_{i0}+f_{i1}\tilde Q_i+\ldots+f_{it}\tilde Q_i^t
\]
be the $\tilde Q_i$-expansion of $f$. For each $f_{ij}\neq 0$ there exists $k\ll i$ such that
\begin{equation}
\nu_k(f_{ij})=\nu(f_{ij}).\label{eq:nukfij}
\end{equation}
Since $I$ is well ordered, we can choose $k$ to be the smallest element of $I$ such that equality \eqref{eq:nukfij} holds for every $j\in\{0,\dots,t\}$. Let
\[
f_{ij}=f_{jk0}+f_{jk1}\tilde Q_k+\ldots+f_{jks}\tilde Q_k^s
\]
be the $\tilde Q_k$-expansion of $f_{ij}$. We can proceed taking expansions of the ``coefficients". Since at each step the degree of the coefficients of the expansions decreases strictly, we will reach the case where these coefficients belong to $K$. Hence we obtain an expression of the form \eqref{expiadicf}. 
\begin{Obs}\label{propertiesithexpansion} Full $i$-th expansions have the following properties.
\begin{enumerate}
\item Equality $\nu_k(f_{ij})=\nu(f_{ij})$ and its analogues hold at each step of our recursive process.
\item The full $i$-th expansion  is not, in general, unique. However, it is uniquely determined by the tuple
\begin{equation}
\left(k\ \left|\ \tilde Q_k\mbox{ appears in \eqref{expiadicf}}\right.\right),\label{eq:tupleappears}
\end{equation}
where we write the indices $k$ in the tuple \eqref{eq:tupleappears} in the decreasing order.
\end{enumerate}
\end{Obs}

\begin{Obs}
By construction of the full $i$-th expansion we have
\[
\deg\left(\frac{\tilde{\textbf{Q}}^{\lambda_j}}{\tilde Q_i^{\lambda_j(Q_i)}}\right)<\deg(Q_i)\mbox{ for every }j,1\leq j\leq r.
\]
\end{Obs}
\medskip

Let $\textbf{X}=\{X_i\}_{i\in I^*}$ be a set of indeterminates.
\medskip

\noindent{\bf Notation.}  For $i\in I^*$, let $\textbf{X}_{<i}=\{X_{i'}\}_{\substack{i'\in I^*\\i'<i}}$, $\textbf{X}_{\le i}=
\{X_{i'}\}_{\substack{i'\in I^*\\i'\le i}}$, $\textbf{X}_{\ge i}=\{X_{i'}\}_{\substack{i'\in I^*\\i'\ge i}}$.
\medskip

Let 
\[
F\in K[\textbf{X}]:=K[X_i\mid i\in I^*]
\]
This means that there exist $\lambda_1,\ldots,\lambda_s\in \N^{I^*}$ and $b_1,\ldots,b_s\in K$ such that
\begin{equation}
F=\sum_{j=1}^sb_j\textbf{X}^{\lambda_j}.\label{eq:fexpansion}
\end{equation}
We denote by
\begin{equation}
F_{\textbf{Q}}:=F\left(\textbf{Q}\right)=\sum_{j=1}^sb_j\textbf{Q}^{\lambda_j}\mbox{ and
}F_{\tilde{\textbf{Q}}}:=F\left(\tilde{\textbf{Q}}\right)=\sum_{j=1}^sb_j\tilde{\textbf{Q}}^{\lambda_j}.\label{eq:fQexpansion}
\end{equation}

\subsection{Immediate and limit successors and the successively strongly monic condition}\label{Immediatesuccessors}

Fix indices $i<\ell$ in $I$. The following definition is inspired by the work of Julie Decaup \cite{D}, but is more restrictive than hers.
\begin{Def} We say that $Q_\ell$ is an {\bf immediate successor} of $Q_i$ (denoted by $Q_i<_{\text{imm}}Q_\ell$) if
$\ell=i+1$.
\end{Def}
\begin{Obs} It may happen that $\deg\ Q_{i+1}=\deg\ Q_i$.
\end{Obs}
\begin{Def}[J. Decaup \cite{D}] We say that $Q_\ell$ is a {\bf limit successor} of $Q_i$ (denoted by $Q_i<_{\lim}Q_\ell$) if $i\in I^{(q)}$, $\#I^{(q)}=\infty$ and $\ell=\ell_{q+1}$ for some $q$, $1\leq q\leq w$.
\end{Def}
\noindent{\bf Notation.} We will write $Q_i<_{\rm succ}Q_\ell$ to indicate that either $Q_i<_{\text{imm}}Q_\ell$ or $Q_i<_{\lim}Q_\ell$.

For an element $f\in K[x]$, let
\begin{equation}\label{eppansnfoposjjfh}
f=f_0+f_1\tilde Q_i+\ldots+f_r\tilde Q_i^r
\end{equation}
be the $\tilde Q_i$-expansion of $f$. The subset $S_i(f)\subset\{0,\dots,r\}$ is defined by
\[
S_i(f):=\left\{j\in\{0,\dots,r\}\ \left|\ \nu\left(f_j\tilde Q_i^j\right)=\nu_i(f)\right.\right\}.
\]
\medskip

\begin{Obs} The same set $S_i(f)$ can be defined analogously using the $Q_i$-expan-sion instead of the $\tilde Q_i$-expansion; the result does not depend on which of the two types of expansions we use.
\end{Obs}

\noindent{\bf Notation.} We put $\delta_i(f):=\max\ S_i(f)$.

Assume that $Q_i<_{\rm succ}Q_\ell$. Let $Q_\ell=\sum\limits_{j=0}^rq_jQ_i^j$ be the $Q_i$-expansion of $Q_\ell$.
\begin{Def} We say that $Q_\ell$ is {\bf strongly $Q_i$-monic} if $q_r=1$ and $r\in S_i(Q_\ell)$.
\end{Def}

\begin{Obs}\label{delta>1}\begin{enumerate}
\item If $Q_\ell$ is strongly $Q_i$-monic then $r=\delta_i(Q_\ell)$.
\item If $i\ll\ell$ then $\deg_xQ_i<\deg_xQ_\ell$. If, in addition, $Q_\ell$ is strongly $Q_i$-monic we obtain that 
$\delta_i(Q_\ell)>1$.
\item In general, it is known (and not hard to prove) that if $i\ll\ell$ then $\delta_i(Q_\ell)$ is a power of $\text{char}\ Kv$ if $\text{char}\ Kv>0$ and is equal to 1 if $\text{char}\ Kv=0$ (in the latter case $Q_\ell$ is not strongly $Q_i$-monic). This fact will not be used in the sequel.
\end{enumerate}
\end{Obs}

\begin{Obs}\begin{enumerate}
\item If $Q_i<_{\text{imm}}Q_\ell$, then $Q_\ell$ is strongly $Q_i$-monic.
\item The property of $Q_\ell$ being strongly $Q_i$-monic does not depend on the particular  choice of $Q_i$ and $Q_\ell$, only on their respective degrees.
\end{enumerate}
\end{Obs}
\begin{Def}\label{stronglymonis}
 We say that $\bf Q$ is {\bf successively strongly monic} if for each pair $(i,\ell)$ as above with $\ell\in I^*$, $Q_\ell$ is strongly $Q_i$-monic.
\end{Def}

\subsection{A description of the $\VR_K$-algebra $\VR_L$ in terms of generators and relations}\label{Characteri1}

The purpose of this section is to give a description of the $\VR_K$-algebra $\VR_L$ in terms of generators and relations.

Consider the natural maps $\VR_K[\textbf{X}]\overset{\bf e}\lra K[x]\overset{\text{ev}_\eta}\lra L$, $f\mapsto f_{\tilde{\bf
Q}}\mapsto f_{\tilde{\bf Q}}(\eta)$. By Proposition \ref{GenerationofOLoverOK}, we have
$(\text{ev}_\eta\circ{\bf e})(\VR_K[\textbf{X}])=\VR_L$.
\medskip

\noindent{\bf Notation.} For $f\in K[\textbf{X}]$, let $\mu_0(f)$ denote the minimum of the values of the coefficients of
$f$:
\begin{equation}\label{eq:mu0=nui}
\mu_0(f)=\min\limits_{j\in\{1,\dots,s\}}\left\{v\left(b_j\right)\right\}
\end{equation}
in the notation of \eqref{eq:fexpansion}.

Let $\mathcal I=\text{Ker}(\text{ev}_\eta\circ{\bf e})$, so that
\[
\VR_L\simeq \VR_K[\textbf{X}]/\mathcal I.
\]
Below we will give an explicit description of $\mathcal I$ in terms of $\mathbf{\tilde Q}$. More precisely, this description will be given in terms of full $i$-th expansions of $Q_\ell$ for certain pairs $\ell,i\in I$, $Q_i<_{\rm succ}Q_\ell$.

Recall the notation $i_{\text{max}}=\max I$ and $g=Q_{i_{\text{max}}}=\tilde Q_{i_{\text{max}}}$. Let $\ell,i$ be as above and consider a full $i$-th expansion
\begin{equation}
\tilde Q_\ell=\sum\limits_ {j=1}^rb_{\ell ij}\tilde{\textbf{Q}}^{\lambda_j}.\label{eq:fulliexpofQtildel}
\end{equation}
Fix an element $b_{\ell i}\in K$ such that
\begin{equation}
v\left(b_{\ell i}\right)=-\nu_i\left(\tilde Q_\ell\right)=-\min\limits_jv\left(b_{\ell ij}\right).\label{eq:minusthevalue}
\end{equation}
In the case when $\ell\ne i_{\max}$, we will choose $b_{\ell i}$ of the following special form. Since in this case $Q_\ell$ is assumed to be strongly $Q_i$-monic, after renumbering the monomials in \eqref{eq:fulliexpofQtildel} we may assume that $\tilde{\textbf{Q}}^{\lambda_1}$ is divisible by $\tilde Q_i$ but not by any $\tilde Q_{i'}$ with $i'<i$ and that $vb_{\ell i1}=\nu_i(\tilde Q_{\ell})$. Put $b_{\ell i}:=\frac1{b_{\ell i1}}$.

Note that, in all cases, we have
\begin{equation}
v\left(b_{\ell i}\right)>0\mbox{ whenever }\ell\in I^*.\label{eq:vbli>0}
\end{equation}
Let $Q_{\ell i}=b_{\ell i}\sum\limits_jb_{\ell ij}{\textbf{X}}^{\lambda_j}$.

\begin{Obs} By definition of $b_{\ell i}$, we have $\mu_0(Q_{\ell i})=0$. In particular,
\[
Q_{\ell i}\in\VR_K[\bf X].
\]
\end{Obs}

In order to simplify the notation and make our constructions more rigid, we will require these full $i$-th expansions to satisfy certain additional conditions that will be spelled out below.
\medskip

Recall that for $\ell\in I$ we denote by $q_\ell$ the unique element of $\{1,\dots,w\}$ such that $\ell\in I^{(q_\ell)}$.

\begin{Def} Consider elements $i,\ell\in I$ such that
\begin{equation}
Q_i<_{\rm succ}Q_\ell.\label{eq:successor}
\end{equation}
We say that the pair $(\ell,i)$ is {\bf neat} if the following implication holds: if $q_i<q_\ell$ (hence $q_\ell=q_i+1$) and $\#I^{(q_i)}=\#I^{(q_\ell)}=\infty$ then
$i-\ell_{q_i}=\ell-\ell_{q_\ell}$. The {\bf level} of the pair $(\ell,i)$ is the integer
\begin{displaymath}
s(\ell,i)=\left\{
\begin{array}{ll}
i-\ell_{q_i}&\text{\,if }\ell=i+1\\
\max\{i-\ell_{q_i},\ell-\ell_{q_\ell}\}&\text{\,if }q_i<q_\ell\mbox{ and }\#I^{(q_i)}=\infty.
\end{array}
\right.
\end{displaymath}
\end{Def}
\medskip

Fix a polynomial $f\in K[x]\setminus\{0\}$ and consider a full $i$-th expansion \eqref{expiadicf} of $f$. Recall that by definition of full $i$-th expansions, for each $q\in\{1,\dots,w\}$ there exists at most one $i\in I^{(q)}$ such that $\tilde Q_i$ appears in \eqref{expiadicf}. If such an $i$ does exist, we will denote it by $j(f,q)$. If there is no such $i$, we will write $j(f,q)=-\infty$; we adopt the convention that $-\infty<I$.

\begin{Def}\label{neatdefinitionQ} We say that the full $i$-th expansion \eqref{expiadicf} is {\bf neat} if there exists $s\in\N$ satisfying
\[
j(f,q)-\ell_q=s\text{ for all }q\in\{1,\dots,w\}\text{ such that }\#I^{(q)}=\infty\text{ and }j(f,q)>-\infty
\]
 (recall that we have $j(f,q)-\ell_q=0$ whenever $\#I^{(q)}=1$ and $j(f,q)>-\infty$).

In this situation, we will say that \eqref{expiadicf} is {\bf neat of level} $s$, provided that the set of indices $q$ with
$\#I^{(q)}=\infty$ and $j(f,q)>-\infty$ appearing above is not empty. If the set of such $q$ is empty, we will say that $f$ is {\bf neat of level} 0.
\end{Def}

\begin{Obs} It is possible for $f$ to be of level 0 even if the set of indices $q$ appearing in the above definition is not empty.
\end{Obs}

For neat pairs $(\ell,i)$, $i,\ell\in I$, one can construct, recursively in the level $s(\ell,i)$, a neat full $i$-th expansion
\begin{equation}
\tilde Q_\ell=\sum\limits_ {j=1}^rb_{\ell ij}\tilde{\textbf{Q}}^{\lambda_j}\label{eq:fulliexpofQtildelbis}
\end{equation}
of $\tilde Q_\ell$. The neatness condition guarantees the uniqueness of such an expansion for each neat pair $(\ell,i)$. 

Let $(\ell,i)$ be a neat pair and fix a full $i$-th expansion \eqref{eq:fulliexpofQtildelbis} of $\tilde Q_\ell$. Let $s:=s(\ell,i)$. Let S$_{\ell,i}$ denote the  statement that  the expansion \eqref{eq:fulliexpofQtildelbis} is neat of level $s$.

\begin{Def}\label{neat} We say that $\bf{\tilde Q}$ is {\bf neat} if S$_{\ell,i}$ holds for all the neat pairs $(\ell,i)$.
\end{Def}

\begin{Lema}\cite[Lemma 4.3]{NS20251} \label{lemaneat} After replacing $I^*$ by a suitable subset $\tilde I^*$, it is possible to choose an expansion \eqref{eq:fulliexpofQtildel} for each neat pair $(\ell,i)$ in such a way that $\bf{\tilde Q}$ becomes neat.
\end{Lema}

From now on, we will assume that $\bf{\tilde Q}$ is  neat.

Let
\begin{equation}
\mathcal I_1=\left(\left.b_{\ell i}X_\ell-Q_{\ell i}\ \right|\  i,\ell\in I^*, (\ell,i)\text{ is neat}\right)\VR_K[\bf X]
\end{equation}
and
\begin{equation}
\mathcal I_2=\left(\left.Q_{i_{\max}i}\ \right|\ i\in I^*,Q_i<_{\rm succ}Q_{i_{\max}}\right)\VR_K[\bf X].
\end{equation}

The following description of $\mathcal I$ is the main result of \cite{NS20251}.
\begin{Prop}\cite[Theorem 5.18]{NS20251}\label{generatorsCalI} We have
\begin{equation}
\mathcal I=\mathcal I_1+\mathcal I_2.\label{eq:I1+I2} 
\end{equation}
\end{Prop}

\section{The module of K\"ahler differentials}\label{Kahler}

Let $\{f_j\}_{j\in J}$ be a set of generators of $\mathcal I$. We have
\begin{equation}
\Omega_{\VR_L/\VR_K}\simeq\frac{\bigoplus\limits_{i\in I^*}\VR_LdX_i}{(df_j\mid j\in J)}\label{eq:generatorsofJ}
\end{equation}
(see \cite{Mat}, Chapter 10, Section 26, (26.E), Example 1 on pp. 183--184 and (26.I), Theorem 58 on pp. 187--188 --- the second fundamental exact sequence --- where we take $A=\VR_K[\bf X]$, $B=\VR_L$ and $\mathfrak m=\mathcal I$).

Theorem \ref{mainthm} follows from Proposition \ref{generatorsCalI} and \eqref{eq:generatorsofJ}.

\subsection{Expressing $\Omega_{\VR_L/\VR_K}$ as a direct limit} One of the main goals of this paper is to give an explicit description of $\Omega_{\VR_L/\VR_K}$ and its annihilator in terms of key polynomials and, in particular, to characterize the situation when
\begin{equation}
\Omega_{\VR_L/\VR_K}=(0).\label{eq:omega=0}
\end{equation}
To this end, we express $\Omega_{\VR_L/\VR_K}$ as the direct limit of its submodules generated by $d\bf X_{\le\ell}$ for various $\ell\in I^*$.
\medskip

\noindent{\bf Notation.} For $\ell\in I^*$, define the ideals $\mathcal I_{1,<\ell}\subset\mathcal I_{1\ell}\subset\mathcal I_1$ and
$\mathcal I_{2,<\ell}\subset\mathcal I_{2\ell}\subset\mathcal I_2$ by
\begin{equation}
\mathcal I_{1\ell}=\left(\left.b_{\ell'i}X_{\ell'}-Q_{\ell'i}\ \right|\ i,\ell'\in I^*,\ell'\le\ell,Q_i<_{\text{succ}}Q_{\ell'}\right),\label{eq:defI1l}
\end{equation}
\begin{equation}
\mathcal I_{1,<\ell}=\left(\left.b_{\ell'i}X_{\ell'}-Q_{\ell'i}\ \right|\ i,\ell'\in I^*,\ell'<\ell,Q_i<_{\text{succ}}Q_{\ell'}\right),\label{eq:defI1<l}
\end{equation}\begin{equation}
\mathcal I_{2\ell}=\left(\left.Q_{i_{\max}i}\ \right|\ i\in I^*,i\le\ell,Q_i<_{\text{succ}}Q_{i_{\max}i}\right)\label{eq:defI2l}
\end{equation}
and
\begin{equation}
\mathcal I_{2,<\ell}=\left(\left.Q_{i_{\max}i}\ \right|\ i\in I^*,i<\ell,Q_i<_{\text{succ}}Q_{i_{\max}}\right).\label{eq:defI2<l}
\end{equation}

Put $\Omega^*:=\frac{\bigoplus\limits_{i\in I^*}\VR_LdX_i}{\mathcal J_1}$.

For $\ell\in I^*$, let $\mathcal I_\ell:=\mathcal I_{1\ell}+\mathcal I_{2\ell}$, $\mathcal I_{<\ell}:=\mathcal I_{1,<\ell}+\mathcal I_{2\ell}$,
\begin{displaymath}
\begin{array}{ccc}
\mathcal J_{1\ell}:=(df\ |\ f\in\mathcal I_{1\ell}),&\mathcal J_{2\ell}:=(df\ |\ f\in\mathcal I_{2\ell}),&\mathcal J_\ell:=(df\ |\ f\in\mathcal I_\ell),\\
[8pt]\mathcal J_{1,<\ell}:=(df\ |\ f\in\mathcal I_{1,<\ell}),&\mathcal J_{2,<\ell}:=(df\ |\ f\in\mathcal I_{2,<\ell}),&\mathcal J_{<\ell}:=(df\ |\ f\in\mathcal I_{<\ell}).
\end{array}
\end{displaymath}
Let $\Omega_\ell^*:=\frac{\bigoplus\limits_{i\le\ell}\VR_LdX_i}{\mathcal J_{1\ell}}$, $\Omega^*_{<\ell}:=\frac{\bigoplus\limits_{i<\ell}\VR_LdX_i}{\mathcal J_{1,<\ell}}$,
$\Omega_\ell:=\frac{\bigoplus\limits_{i\le\ell}\VR_LdX_i}{\mathcal J_\ell}$ and $\Omega_{<\ell}:=\frac{\bigoplus\limits_{i<\ell}\VR_LdX_i}{\mathcal J_{<\ell}}$. We have
\begin{eqnarray}
\Omega_{\VR_L/\VR_K}&=&\lim\limits_{\substack{\longrightarrow\\\ell\in I^*}}\Omega_\ell,\\
\Omega_{<\ell}&=&\lim\limits_{\substack{\longrightarrow\\i<\ell}}\Omega_i,\\
\Omega^*&=&\lim\limits_{\substack{\longrightarrow\\\ell\in I^*}}\Omega^*_\ell\mbox{ and}\\
\Omega^*_{<\ell}&=&\lim\limits_{\substack{\longrightarrow\\i<\ell}}\Omega^*_i.
\end{eqnarray}

\section{Multilines}\label{multilies}

Fix a non-negative integer $w$ and a well ordered set $I^*$ together with a decomposition
\begin{equation}
I^*=\coprod\limits_{q=1}^wI^{(q)},\label{eq:I*decomposition}
\end{equation}
where the $I^{(q)}$ are segments in $I^*$ with $I^{(1)}<I^{(2)}<\dots<I^{(w)}$ and each $I^{(q)}$ is either a singleton or is order-isomorphic to the set $\N$ of natural numbers.
\medskip

\noindent{\bf Notation.} Let $\ell_q=\min\ I^{(q)}$, $q\in\{1,\dots,w\}$.

Let $Y=\{Y_i\}_{i\in I^*}$ be independent variables.

For $q\in\{1,\dots,w\}$, we will write ${\bf Y}^{(q)}=\{Y_i\}_{i\in I^{(q)}}$.
\medskip

In what follows, we will study quotients of the free $\VR_L$-module $\bigoplus\limits_{i\in I^*}\VR_LY_i$ by relations of certain prescribed types.

\begin{Def} Consider an expression of the form
\begin{equation}
\sum\limits_{q=1}^wh_qY_{i_q},i_q\in I^{(q)},h_q\in\VR_L.\label{eq:neatform}
\end{equation}
We say that the form \eqref{eq:neatform} is {\bf neat} if for each $q,q'\in\{1,\dots,w\}$ such that
\begin{equation}\label{eq:level}
\#I^{(q)}=\#I^{(q')}=\infty\mbox{ and }h_q\neq 0\neq h_{q'}
\end{equation}
we have $i_q-\ell_q=i_{q'}-\ell_{q'}$. This common value of $i_q-\ell_q$ is called the {\bf level} of the form \eqref{eq:neatform}. If for every $q$ \eqref{eq:level} is not satisfied, then we say that the level \eqref{eq:neatform} is zero. 
\end{Def}

\begin{Def} A {\bf free multiline over} $\VR_L$ {\bf of width} $w$ is an $\VR_L$-module of the form $M=\frac{\bigoplus\limits_{i\in I^*}\VR_L Y_i}{(\bf H)}$, with
\begin{equation}
\textbf H=\left\{H_i\ \left|\ i\in I^*,\exists q\in\{1,\dots,w\}\text{ such that }i,i+1\in I^{(q)}\right.\right\},\label{eq:Hi}
\end{equation}
where
\[
H_i=b_{i+1,i}Y_{i+1}-Y_i-R_{i+1,i},
\]
$b_{i+1,i}\in\mathfrak m_L$ and
\begin{equation}
R_{i+1,i}\in\bigoplus\limits_{q'=1}^{q-1}\VR_LY_{i_{q'}}\label{eq:Rli}
\end{equation}
is neat.
\end{Def}

\begin{Def}  A {\bf multiline over} $\VR_L$ {\bf of width at most} $w$ is an $\VR_L$-module of the form $M=\frac{M^*}{{\bf H}^*}$ , where $M^*$ is a free multiline of width $w$ and ${\bf H}^*$ an $\VR_L$-submodule of $M^*$ having a set of generators $\left\{H^*_j\right\}_{j\in J}$ such that $H^*_j$ is neat for all $j\in J$.
\end{Def}

Below we will show that ${\bf H}^*$ is itself a free multiline with
\[
\text{width}\left({\bf H}^*\right)\le\text{width}\left(M^*\right)=w.
\]
\begin{Def}\label{deffreelygenerate} We will say that ${\bf Y}^{(1)},\dots,{\bf Y}^{(w)}$ form a {\bf set of generators} or {\bf a basis} of the multiline $M$. If $M$ is free, we will say that ${\bf Y}^{(1)},\dots,{\bf Y}^{(w)}$ {\bf freely generate} $M$.
\end{Def}

\begin{Obs} Let $M$ be a multiline of width at most $w$. Take an element $q\in\{1,\dots,w\}$ and set
\[
\mathcal M_q=\bigoplus\limits_{\substack{i\in I^*\\i<\ell_{q+1}}}\VR_LY_i.
\]
Let
\begin{displaymath}
\begin{array}{rclcrcl}
{\bf H}(q)&=&{\bf H}\bigcap\mathcal M_q,&&M^*(q)&=&\displaystyle\frac{\mathcal M_q}{{\bf H}(q)},\\[10pt]
 {\bf H}^*(q)&=&{\bf H}^*\bigcap M^*(q)&\mbox{and}& M(q)&=&\displaystyle\frac{M^*(q)}{{\bf H}^*(q)}
\end{array}.
\end{displaymath}

We obtain a commutative diagram
\begin{equation}\label{eq:nestedmultilines}\xymatrix{
{\bf H}^*(1)&\hookrightarrow&{\bf H}^*(2)&\hookrightarrow&\dots&\hookrightarrow&{\bf H}^*(w)\\
\downarrow&&\downarrow&&\dots&&\downarrow\\
M^*(1)&\hookrightarrow&M^*(2)&\hookrightarrow&\dots&\hookrightarrow&M^*(w)\\
\downarrow&&\downarrow&&\dots&&\downarrow\\
M(1)&\hookrightarrow&M(2)&\hookrightarrow&\dots&\hookrightarrow&M(w)}
\end{equation}
where the middle row consists of free multilines of widths at most $1,2,\dots,w$, respectively, the top row consists of submodules of the $M^*(q)$ and the columns are exact sequences.
\end{Obs}

\begin{Obs} For $q\in\{1,\dots,w-1\}$, the module $M(q)$ is the submodule of $M$ generated by $\left({\bf Y}^{(q')}\right)_{1\le q'\le q}$.
\end{Obs}

We have associated to each multiline $M$ a filtration
\begin{equation}
M=M(w)\supsetneqq M(w-1)\supsetneqq\dots\supsetneqq M(1)\supsetneqq M(0)=(0)\label{eq:filtrationgeneric}
\end{equation}
by $\VR_L$-modules satisfying the following conditions for each $q\in\{1,\dots,w\}$:
\begin{enumerate}
\item $M( q)=\lim\limits_{\substack{\longrightarrow\\i\in I^{( q)}}}(M( q-1)+\VR_LY_i)$, where
\[
Y_i\in M( q-1)+\mathfrak m_LY_{i+1}
\]
whenever $i,i+1\in I^{(q)}$.
\item
\begin{equation}
\frac{M(q)}{M(q-1)}\cong\bigcup\limits_{i\in I^{(q)}}\VR_L\bar Y_i,\label{eq:isomorphismlimit}
\end{equation}
where $\bar Y_i$ denotes the natural image of $Y_i$ in $\frac{M(q)}{M(q-1)}$.
\end{enumerate}
\begin{Def}\label{quasivectorspace} The filtration \eqref{eq:filtrationgeneric} of $M$ by multilines $M(q)$ will be referred to as {\bf a standard filtration} of the multiline $M$.
\end{Def}
\begin{Obs} The same $\VR_L$-module $M$ can have several multiline structures and even the integer $w$ is not uniquely determined by $M$ (this point will be explained in more detail later).
\end{Obs}
\begin{Def} Let $M$ be an $\VR_L$-module that admits a multiline structure. The {\bf width} of $M$ is the minimal integer $w$ for which such a multiline structure of width at most $w$ exists.
\end{Def}
By convention, we will sometimes write $M(w)$ for $M$.

\begin{Exa} The $\VR_L$-modules $\Omega_{\VR_L/\VR_K}$ and $\Omega^*$ (see the previous subsection) are multilines. The multiline $\Omega_{\VR_L/\VR_K}$ is never free and $\Omega^*$ may or may not be free; it is always free whenever $w=1$. In both cases, the decomposition of the set $I^*$ as in \eqref{eq:I*decomposition} is described at the Introduction. Both $\Omega_{\VR_L/\VR_K}$ and $\Omega^*$  are homomorphic images of the free multiline
\begin{equation}
\frac{\bigoplus\limits_{i\in I^*}\VR_LdX_i}{\left(\left.b_{i+1,i}dX_{i+1}-dQ_{i+1,i}\ \right|\ \exists q\in\{1,\dots,w\}\text{ such that }i,i+1\in I^{(q)}\right)}.\label{eq:commonfreemultiline}
\end{equation}
The submodule ${\bf H}^*$ of \eqref{eq:commonfreemultiline} is
\[
\left(\left.b_{\ell i}dX_\ell-dQ_{\ell i}\ \right|\ i\ll\ell,(\ell,i)\text{ is neat}\right)
\]
in the case of $\Omega^*$ and
\[
\left(\left.b_{\ell i}dX_{\ell}-dQ_{\ell i}\ \right|\ i\ll\ell,(\ell,i)\text{ is neat}\right)+\mathcal J_2
\]
in the case of $\Omega_{\VR_L/\VR_K}$.

Other multilines of interest to us will be the $\VR_L$-modules of the form $a\Omega^*$ and $a\Omega_{\VR_L/\VR_K}$, where $a\in\VR_L$.
\end{Exa}

\section{A description of $\text{Ann}(M)$}

Fix a multiline $M$ of width at most $w$. By definition,
\begin{equation}
\text{width}(M)\le w.\label{eq:width(M)<w}
\end{equation}
In general, the inequality \eqref{eq:width(M)<w} may be strict.

In this section we will describe $\text{Ann}(M)$ and the width of $M$. In particular, we will give a necessary and sufficient condition for equality to hold in \eqref{eq:width(M)<w}.
\medskip

Write ${\bf H}^*=\left\{H^*_j\right\}_{j\in J}$, where $J$ is a suitable index set. For  $j\in J$, write
\[
H^*_j=\sum\limits_{q=1}^wh^*_{qj}Y_{m_{qj}},
\]
where $m_{qj}\in{I^{(q)}}$.

We define the fractional ideals
\begin{equation}
I_{\gamma q}\subset I_{\alpha q}\subset L,\quad q\in\{1,\dots,w\}
\end{equation}
as follows. Put
\begin{equation}
I_{\alpha q}:=\bigcup\limits_{i\in I^{(q)}}\frac1{\prod\limits_{\substack{i'\in I^{(q)}\\i' <i}}b_{i'+1,i'}}\VR_L\label{eq:Ialphaq}
\end{equation}
and
\begin{equation}
I_{\gamma q}=\bigcup\limits_{j\in J}\frac{h^*_{qj}}{\prod\limits_{\substack{i'\in I^{(q)}\\\ell_q\le i'<m_{qj}}}b_{i'+1,i'}}\VR_L.\label{eq:Ibetaq}
\end{equation}
By convention, we take $I_{\alpha q}=\VR_L$ and $I_{\gamma q}=\bigcup\limits_{j\in J}(h^*_{qj})\VR_L$ if $\#I^{(q)}=1$.

\begin{Obs} We have
\begin{equation}
I_{\alpha q}\cong\frac{M ^*(q)}{M^*(q-1)}.\label{eq:IalphaqM(q)M(q-1)}
\end{equation}
If $\tilde Y_i$ denotes the natural image of $Y_i$ in $\frac{M^*(q)}{M^*(q-1)}$, $i\in I_{<\ell_{q+1}}\setminus I_{<\ell_q}$, then the isomorphism \eqref{eq:IalphaqM(q)M(q-1)} is given by sending 1 to $\tilde Y_{\ell_q}$, so that $\tilde Y_i$ gets sent to $\frac1{\prod\limits_{\substack{i'\in I^{(q)}\\i'<i}}b_{i'+1,i'}}$ for $i\in I^{(q)}$.
\end{Obs}

\begin{Def} {\bf The initial index} of ${\bf H}^*$ (denoted by $q\left({\bf H}^*\right)$) is the greatest element
$q\in\{1,\dots,w\}$ such that
\begin{equation}
I_{\gamma q'}:I_{\alpha q'}\subset I_{\gamma q}:I_{\alpha q}\quad\text{ for all }\ q'\in\{1,\dots,w\}.
\end{equation}
{\bf The initial form} of ${\bf H}^*$ (denoted by $\IN\ {\bf H}^*$) is the ideal $I_{\gamma q\left({\bf H}^*\right)}:I_{\alpha q\left({\bf H}^*\right)}\subset\VR_L$.
\end{Def}
\begin{Obs}\begin{enumerate}
\item We claim that ${\bf H}^*$ is itself a free multiline. Indeed, a standard filtration 
\begin{equation}
{\bf H}^*={\bf H}^{*(1)}\supsetneqq{\bf H}^{*(2)}\supsetneqq\dots\supsetneqq{\bf
H}^{*(u)}\supsetneqq(0)\label{eq:filtration}
\end{equation}
of ${\bf H}^*$ can be constructed by a form of Gram--Schmidt process. Namely, let ${\bf H}^{*(2)}:={\bf H}^*\bigcap\sum\limits_{\substack{q=1\\q\ne q\left({\bf
H}^*\right)}}^w\sum\limits_{i\in I^{(q)}}\VR_LY_i$. The isomorphism \eqref{eq:IalphaqM(q)M(q-1)} induces an isomorphism
\begin{equation}
I_{\gamma q}\cong\frac{{\bf H}^{*(1)}}{{\bf H}^{*(2)}}.\label{eq:IgammaqH(1)HM(2)}
\end{equation}
Replace ${\bf H}^*$ by ${\bf H}^{*(2)}$ and repeat the procedure to construct ${\bf H}^{*(3)}$, and so on.
\item All the integers $q\left({\bf H}^*\right),q\left({\bf H}^{*(2)}\right),\dots,q\left({\bf H}^{*(u)}\right)$ are distinct.
\item This shows that an $\VR_L$-submodule of a free multiline of width $w$ is itself a free multiline of width at most $w$.
\end{enumerate}
\end{Obs}
\begin{Obs}\label{Descriptionannihilator} Let $u'$ be the greatest element of $\{1,\dots,u\}$ such that
\begin{equation}
\IN\ {\bf H}^{*(u')}=\VR_L\label{eq:defofu'}
\end{equation}
(if no such element of $\{1,\dots,u\}$ exists, put $u'=0$). For an element $a\in\VR_L$ let $u'(a)$ be the greatest element of $\{1,\dots,u\}$ such that
\begin{equation}
a\in I_{\gamma q\left({\bf H}^{*(u'(a))}\right)}:I_{\alpha q\left({\bf H}^{*(u'(a))}\right)},\label{eq:defofu'(a)}
\end{equation}
so that $u'=u'(1)$. Again, if no such element of $\{1,\dots,u\}$ exists, put $u'(a)=0$. Note that \eqref{eq:defofu'(a)} implies that
\begin{equation}
a\in I_{\gamma q\left({\bf H}^{*(u'')}\right)}:I_{\alpha q\left({\bf H}^{*(u'')}\right)},\label{eq:defofu'(a)bis}
\end{equation}
for all $u''\in\{1,\dots,u'(a)\}$.
\begin{enumerate}
\item  We have $\text{width}(M)=w-u'$;
\[
M\cong\frac{\bigoplus\limits_{q\in\{1,\dots,w\}\setminus\left\{q\left({\bf H}^{*(1)}\right),\dots,q\left({\bf
H}^{*(u')}\right)\right\}}\left(\sum\limits_{i\in I^{(q)}}\VR_LY_i\right)}{{\bf H}^{*(u'+1)}}.
\]
In particular, $M=(0)$ if and only if $u=u'=w$ (by definition of $u'$, this equality is equivalent to saying that 
$\IN\ {\bf H}^{*(u)}=\VR_L$).
\item  We have $\text{width}(aM)=w-u'(a)$. There are homomorphisms
\[
M\overset\pi\twoheadrightarrow\frac{\bigoplus\limits_{q\in\{1,\dots,w\}\setminus\left\{q\left({\bf H}^{*(1)}\right),\dots,q\left({\bf
H}^{*(u'(a))}\right)\right\}}\left(\sum\limits_{i\in I^{(q)}}\VR_LY_i\right)}{a^{-1}{\bf H}^{*(u'(a)+1)}}\overset\phi{\xrightarrow{\sim}}aM
\]
with the following properties. If
\begin{equation}
i\in I^{(q)}\text{ with }q\in\{1,\dots,w\}\setminus\left\{q\left({\bf H}^{*(1)}\right),\dots,q\left({\bf H}^{*(u'(a))}\right)\right\},\label{eq:qnotinitial}
\end{equation}
then the surjective map $\pi$ maps the class of $Y_i$ to the class of $Y_i$. If
\begin{equation}
i\in I^{(q)}\text{ with }q\in\left\{q\left({\bf H}^{*(1)}\right),\dots,q\left({\bf H}^{*(u{'(a)})}\right)\right\},\label{eq:qinitial}
\end{equation}
let $u''$ be the element of $\{1,2,\dots,u'(a)\}$ such that $q=q\left({\bf H}^{*(u'')}\right)$. By \eqref{eq:defofu'(a)bis}, there exists an element
\begin{equation}
\begin{split}
y_i=&\sum\limits_{q'\in\{1,\dots,w\}\setminus\left\{q\left({\bf H}^{*(1)}\right),\dots,q\left({\bf H}^{*(u'(a))}\right)\right\}}\tilde h^*_{q'i}Y_{ m_{q'i}}\in\\
&\in\bigoplus\limits_{q'\in\{1,\dots,w\}\setminus\left\{q\left({\bf H}^{*(1)}\right),\dots,q\left({\bf H}^{*(u'(a))}\right)\right\}}\left(\sum\limits_{i'\in
I^{(q')}}\VR_LY_{i'}\right)
\end{split}
\label{eq:imageofYi}
\end{equation}
such that
\[
Y_i+y_i\in a^{-1}{\bf H}^{*(u'')}.
\]
Pick and fix one such element. Let $\pi$ map the class of $Y_i$ to the class of $-y_i$.

The isomorphism $\phi^{-1}$ maps the class of $aY_i$ to the class of $Y_i$ if \eqref{eq:qnotinitial} holds and to the class of  $-y_i$ if \eqref{eq:qinitial} holds.
\item If $u<w$, then  $\text{Ann}_{\VR_L}M=(0)$. If $u=w$, then
\[
\text{Ann}_{\VR_L}M=I_{\gamma q\left({\bf H}^{*(u)}\right)}:I_{\alpha q\left({\bf H}^{*(u)}\right)}.
\]
\end{enumerate}
\end{Obs}

\subsection{The invertibility hypotheses and a criterion for $\Omega_{\VR_L/\VR_K}$ to be $(0)$}\label{invertivilie} In the special case of the multiline $\Omega_{\VR_L/\VR_K}$, the relations ${\bf H}^*$ are ``almost triangular", without applying any Gram--Schmidt process. Because of this, we can give a sufficient and a necessary condition for the equality
\begin{equation}
\Omega_{\VR_L/\VR_K}=(0)\label{eq:Omega0repeat}
\end{equation}
to hold that is more explicit than the criterion of Remark \ref{Descriptionannihilator} (1). To this end, we will introduce a key condition, called the invertibility hypothesis.

Take an element $q\in\{1,\dots,w\}$. Define the fractional ideal $I_{\beta q}$ by
\begin{equation}
I_{\beta q}=\left(\left.\left(\frac{\partial Q_{\ell,i}}{\partial X_i}\right)_{\bf\tilde Q}\frac1{\prod\limits_{\substack{i'\in I^{(q)}\\i' <i}}b_{i'+1,i'}}\ \right|\ i\in
I^{(q)},i\ll\ell,(\ell,i)\text{ neat}\right)\label{eq:betal}
\end{equation}
(recall that, by convention, $I^{(w+1)}=\{\ell_{w+1}\}$ is a singleton and $Q_{\ell_{w+1}}=g$). By definitions, we have
\[
I_{\beta q}\subset I_{\gamma q}\subset I_{\alpha q};
\]
the inclusions may or may not be strict.

\begin{Def} We say that the {\bf invertibility hypothesis is satisfied at level} $q$ if
\begin{equation}
I_{\alpha q}=I_{\beta q}.\label{eq:Ialphal=Ibetal}
\end{equation}
We say that the {\bf invertibility hypothesis holds for the extension} $(L/K,v)$ if it holds at level $q$ for all
$q\in \{1,\dots,w\}$.
\end{Def}

\begin{Obs}\label{invertibilityhypothesis} For the purposes of this remark, assume that $\ell_{q+1}$ is not a limit ordinal. Then the set of indices $i'$ appearing in \eqref{eq:Ialphaq} and \eqref{eq:betal} (namely, those with
$\ell_q<i'\le i$) is empty, so $I_{\alpha q}=\VR_L$. Therefore, in this case the following conditions are equivalent:
\begin{enumerate}
\item the invertibility hypothesis holds at level $q$
\item
\begin{equation}
\nu\left(\left(\frac{\partial Q_{\ell_{q+1},\ell_q}}{\partial X_{\ell_q}}\right)_{\bf{\tilde Q}}\right)=0\label{eq:invertibility}
\end{equation}
\item
\begin{equation}
\nu_{\ell_q}\left(\left(\frac{\partial Q_{\ell_{q+1},\ell_q}}{\partial X_{\ell_q}}\right)_{\bf{\tilde Q}}\right)=0.\label{eq:invertibilitynonlimit}
\end{equation}
\item $p\nmid S_{\ell_q}(Q_{\ell_{q+1}})$. 
\end{enumerate}
Indeed, the equivalence (1)$\iff$(2) follows  easily from definitions. Next, note that
$\deg_{X_{\ell_q}}Q_{\ell_{q+1},\ell_q}=\deg_{Q_{\ell_q}}Q_{\ell_{q+1}}$, hence
\[
\deg_{X_{\ell_q}}\frac{\partial Q_{\ell_{q+1},\ell_q}}{\partial X_{\ell_q}}<\deg_{Q_{\ell_q}}Q_{\ell_{q+1}}.
\]
It follows that $\nu_{\ell_q}\left(\left(\frac{\partial Q_{\ell_{q+1},\ell_q}}{\partial X_{\ell_q}}\right)_{\bf{\tilde
Q}}\right)=\nu\left(\left(\frac{\partial Q_{\ell_{q+1},\ell_q}}{\partial X_{\ell_q}}\right)_{\bf{\tilde Q}}\right)$, which implies the equivalence (2)$\iff$(3). Let us show that (3) equivalent to (4). Recall that $S_{\ell_q}(Q_{\ell_{q+1}})$ is the set of those monomials in $\tilde Q_{\ell_q}$ appearing in the $\ell_q$-expansion of $\tilde Q_{\ell_{q+1}}$ that have value 0, while the remaining monomials have strictly positive values. If (4) holds then $\left(\frac{\partial Q_{\ell_{q+1},\ell_q}}{\partial
X_{\ell_q}}\right)_{\bf{\tilde Q}}$ is a sum of monomials in $\tilde Q_{\ell_q}$ of non-negative values, some of which have value exactly 0 and in all of which the exponent of $\tilde Q_{\ell_q}$ is strictly less than $\deg_{Q_{\ell_q}}Q_{\ell_{q+1}}$. This implies (3). Conversely, if (4) does not hold then $\left(\frac{\partial Q_{\ell_{q+1},\ell_q}}{\partial
X_{\ell_q}}\right)_{\bf{\tilde Q}}$ is a sum of monomials in $\tilde Q_{\ell_q}$ of strictly positive values, so (3) does not hold.
\end{Obs}

Fix a $q\in\{1,\dots,w\}$. If $\#I^{(q)}=\infty$ and $i\in I^{(q)}$, let $\ell(i)$ denote the unique element of $I^{(q+1)}$ such that the pair $(\ell(i),i)$ is neat. If $\#I^{(q)}=1$, put $\ell(\ell_q)=\ell_{q+1}$.

\begin{Obs}\label{invertimpliesmultiple} Assume that the invertibility hypothesis holds at level $q$.
\begin{enumerate}
\item If $\ell_{q+1}$ is a non-limit ordinal, then
\[
b_{\ell_{q+1},\ell_q}dX_{\ell_{q+1}}\equiv\left(\frac{\partial Q_{\ell_{q+1},\ell_q}}{\partial X_{\ell_q}}\right)_{\bf{\tilde
Q}}dX_{\ell_q}\mod\Omega_{<\ell_q}.
\]
where
\[
\nu\left(\left(\frac{\partial Q_{\ell_{q+1},\ell_q}}{\partial X_{\ell_q}}\right)_{\bf{\tilde Q}}\right)=0
\]
(see Remark \ref{invertibilityhypothesis}). In particular, $dX_{\ell_q}$ is an $\mathfrak m_L$-multiple of $dX_{\ell_{q+1}}$ modulo $\Omega_{<\ell_q}$.
\item Regardless of whether or not $\ell_{q+1}$ is a limit ordinal, every $dX_i$, $i\in I^{(q)}$ is an $\mathfrak m_L$-multiple of $dX_{\ell(i)}$ modulo
$\Omega_{<\ell_q}$.
\end{enumerate}
\end{Obs}

In the sequel,  we adopt the convention $X_{\ell_{w+1}}=dX_{\ell_{w+1}}=0$.
\medskip

\begin{Obs}\label{notDeltabounded} The invertibility hypothesis holds at level $q$ if and only if
\begin{equation}
\forall i\in I^{(q)}\ \exists i'\in I^{(q)},i'>i,\text{ such that }\nu\left(\left(\frac{\partial Q_{\ell(i'),i'}}{\partial X_{i'}}\right)_{\bf{\tilde Q}}\right)\le\sum\limits_{i\le i''<i'}vb_{i''+1,i''}.\label{eq:cofinal}
\end{equation}
If $\#I^{(q)}=1$, we take \eqref{eq:cofinal} to mean
\begin{equation}
\nu\left(\left(\frac{\partial Q_{\ell_{q+1},\ell_q}}{\partial X_{\ell_q}}\right)_{\bf{\tilde Q}}\right)=0,\label{eq:derivinvertible}
\end{equation}
by convention.
\end{Obs}
\medskip

Assume that $\ell_{q+1}$ is a limit ordinal. We will now give a much more explicit reformulation of the condition \eqref{eq:cofinal}. We start with a preliminary lemma. Fix an $i_0\in I^{(q)}$ sufficiently large so that for all $i\in I^{(q)}$, $i\ge i_0$, the quantity $\delta_i\left(Q_{\ell_{q+1}}\right)$ is independent of $i$.

\begin{Obs} If $q<w$ then $\delta_i\left(Q_{\ell_{q+1}}\right)=\deg_{Q_i}Q_{\ell_{q+1}}$ for all $i\in I^{(q)}$ by the strongly monic assumption. This quantity is independent of $i$, so we may take $i_0=\ell_q$. Thus the condition that $\delta_i\left(Q_{\ell_{q+1}}\right)$ be independent of $i\ge i_0$ is only non-vacuous for $q=w$.
\end{Obs}

\noindent{\bf Notation.} Let $\delta:=\delta_{i_0}(Q_{\ell(i_0)})$. Let $\widetilde{\mathcal I}$ denote the ideal of $\VR_K[\bf X]$ defined by $\widetilde{\mathcal I}=(b_{i+1,i}X_{i+1}-Q_{i+1,i}\ |\ i,i+1\in I^{(q')},q'\in\{1,\dots,w\})\VR_K[\bf X]$.

\begin{Lema}\label{vdQli+1vsvdQli} Assume that $\ell_{q+1}$ is a limit ordinal. For all $i,i'\in I^{(q)}$ satisfying $i_0\le i<i'$, we have
\begin{equation}\label{eq:upperbound}
\nu\left(\left(\frac{\partial Q_{\ell_{q+1},i}}{\partial X_i}\right)_{\bf{\tilde Q}}\right)=\nu\left(\left(\frac{\partial Q_{\ell_{q+1},i'}}{\partial X_{i'}}\right)_{\bf{\tilde Q}}\right)+(\delta-1)\sum\limits_{i''=i}^{i'}vb_{i''+1,i''}.
\end{equation}
\end{Lema}
\begin{proof} First, consider the case when $i'=i+1$. By definitons, we have
\begin{equation}
Q_{\ell_{q+1},i}\equiv\frac{b_{\ell_{q+1},i}}{b_{\ell_{q+1},i+1}}Q_{\ell_{q+1},i+1}\mod\widetilde{\mathcal I}\label{eq:Qli+1vsQli}
\end{equation}
(this fact does not depend on the assumption $i\ge i_0$). Assume, in addition, that $i_0\le i$. Then $\delta\in S_i(Q_{\ell_{q+1},i})\bigcap S_i(Q_{\ell_{q+1},i+1})$, so 
\begin{equation}
vb_{\ell_{q+1},i}=vb_{\ell_{q+1},i+1}+\delta vb_{i+1,i}.\label{eq:vblivsvbi+1i}
\end{equation}
Combining this with \eqref{eq:Qli+1vsQli} and the fact that $b_{i+1,i}dX_{i+1}=dX_i\mod\Omega^*_{<\ell_q}$ in $\Omega^*$, we obtain
\[
\nu\left(\left(\frac{\partial Q_{\ell_{q+1},i+1}}{\partial X_{i+1}}\right)_{\bf{\tilde Q}}\right)=\nu\left(\left(\frac{\partial Q_{\ell_{q+1},i}}{\partial X_i}\right)_{\bf{\tilde
Q}}\right)-\delta vb_{i+1,i}+vb_{i+1,i}.
\]
This proves the result in the special case when $i'=i+1$. The case of arbitrary $i'\ge i\ge i_0$ follows from this by induction on $i'-i$. 
\end{proof}

\begin{Cor} For all $i\in I^{(q)}_{\ge i_0}$ we have 
\begin{equation}\label{eq:upperboundcontradicted}
(\delta-1)\sum\limits_{i'=i_0}^ivb_{i'+1,i'}<\nu\left(\left(\frac{\partial Q_{\ell_{q+1},i_0}}{\partial X_{i_0}}\right)_{\bf{\tilde Q}}\right).
\end{equation}
\end{Cor}
\begin{proof} Apply Lemma \ref{vdQli+1vsvdQli} with $i$ replaced by $i_0$ and $i'$ replaced by $i+1$. We obtain
\begin{equation}\label{eq:upperboundbis}
\nu\left(\left(\frac{\partial Q_{\ell_{q+1},i_0}}{\partial X_{i_0}}\right)_{\bf{\tilde Q}}\right)=\nu\left(\left(\frac{\partial Q_{\ell_{q+1},i+1}}{\partial X_{i+1}}\right)_{\bf{\tilde
Q}}\right)+(\delta-1)\sum\limits_{i'=i_0}^{i+1}vb_{i'+1,i'}.
\end{equation}
The corollary follows by observing that $\nu\left(\left(\frac{\partial Q_{\ell_{q+1},i+1}}{\partial X_{i+1}}\right)_{\bf{\tilde Q}}\right)\ge0$ and
\[
vb_{i+2,i+1}>0.
\]
\end{proof}

Assume that $\#I^{(q)}=\infty$. For each $i\in I^{(q)}$, let $\Delta_i$ denote the smallest isolated subgroup of $vL$ containing $b_{i'+1,i'}$ for all $i'\in  I^{(q)}_{\ge i}$.

\begin{Obs}\begin{enumerate}
\item The sequence of subgroups $\left(\Delta_i\right)_{i\in I^{(q)}}$ is non-increasing.
\item The sequence $\left(\Delta_i\right)_{i\in I^{(q)}}$ may or may not stabilize. If it stabilizes, let $\Delta^{(q)}$ denote its stable value.
\item Regardless of whether the sequence $\left(\Delta_i\right)_{i\in I^{(q)}}$ stabilizes or not, there exists an isolated subgroup
$\Delta^{(q)}_-$ that is strictly contained in each $\Delta_i$ and is maximal with respect to this property.
\item If $\left(\Delta_i\right)_{i\in I^{(q)}}$ does not stabilize then $\Delta^{(q)}_-=\bigcap\limits_{i\in I^{(q)}}\Delta_i$.
\item If $\left(\Delta_i\right)_{i\in I^{(q)}}$ stabilizes then $\frac{\Delta^{(q)}}{\Delta^{(q)}_-}\cong\R$.
\end{enumerate}
\end{Obs}

\noindent{\bf Notation.} If $\ell_{q+1}$ is not a limit ordinal, we adopt the convention that $\Delta^{(q)}_-=(0)$.
\medskip

We will now give a much more explicit reformulation of the condition \eqref{eq:cofinal} in terms of the subgroup $\Delta^{(q)}_-$.

In the proof of the following proposition, the condition $\delta>1$ holds automatically whenever $q<w$, since $\mathbf Q$ is assumed to be successively strongly monic (see Remark \ref{delta>1}).

\begin{Prop}\label{notDeltabounded1} The invertibility hypothesis holds at level $q$ if and only if
\begin{equation}
\lim\limits_{\substack{\longrightarrow\\i\in I^{(q)}}}\left(\nu\left(\left(\frac{\partial Q_{\ell_{q+1}i_0}}{\partial X_{i_0}}\right)_{\bf\tilde Q}\right)-
(\delta-1)\sum\limits_{i'=i_0}^{i+1}vb_{i'+1,i'}\mod\Delta^{(q)}_-\right)=0\label{eq:goesto0}
\end{equation}
in $\frac{vL}{\Delta^{(q)}_-}$. This condition is always satisfied whenever $\nu\left(\left(\frac{\partial Q_{\ell_{q+1}i_0}}{\partial X_{i_0}}\right)_{\bf\tilde Q}\right)=0$ (in particular, whenever $\delta=1$).

Again, if $\#I^{(q)}=1$, we take \eqref{eq:goesto0} to mean \eqref{eq:derivinvertible} by convention.
\end{Prop}
\begin{proof} Both statements of the proposition follow from the characterization of the invertibility hypothesis of Remark \ref{notDeltabounded}, Lemma \ref{vdQli+1vsvdQli} (applied with $(i,i')$ replaced by $(i_0,i)$) and the fact that
$\nu\left(\left(\frac{\partial Q_{\ell,i}}{\partial X_i}\right)_{\bf{\tilde Q}}\right)=\nu\left(\left(\frac{\partial Q_{\ell',i}}{\partial
X_i}\right)_{\bf{\tilde Q}}\right)$ for all $i\in I^{(q)}$ and all $\ell,\ell'\in I^{(q+1)}$.
\end{proof}

\begin{Def} For $q,q'\in\{1,\dots,w\}$, $q\le q'$, we say that $[q,q']$ is {\bf an interval of invertibility} if the invertibility hypothesis holds at level $q''$ for all $q''$ with $q\le q''\le q'$.
\end{Def}

\begin{Prop}\label{upperboundonwidth} Let $\left\{\left[q_1,q'_1\right],\left[q_2,q'_2\right],\dots,\left[q_r,q'_r\right]\right\}$ be the set of all the maximal intervals of invertibility, arranged in the increasing order of $q_j$. Then
\[
{\rm width}(\Omega_{\VR_L/\VR_K})\le w-\sum\limits_{j=1}^r\left(q'_j-q_j\right).
\]
\end{Prop}
\begin{proof} By Remark \ref{invertimpliesmultiple}, for each $j\in\{1,\dots,r\}$ we have
\[
\left\{dX_i\ \left|\ i\in I^{(q)},q\in\left[q_j,q'_j\right]\right.\right\}\subset\VR_LdX_{\ell(i)},
\]
so that, in the notation of Definition \ref{deffreelygenerate},
\[
\left\{d{\bf X}^{(q)}\ \left|\ q\in\{1,\dots,w\}\setminus\bigcup\limits_{j=1}^r\left[q_j,q'_j\right]\right.\right\}
\]
form a set of generators of the multiline $\Omega_{\VR_L/\VR_K}$. The result follows.
\end{proof}
\noindent{\bf Notation.} For $q\in\{2,\dots,w\}$, let $T_q$ denote the fractional ideal
\[
T_q:=\bigcup_{\substack{i\in I^{(q-1)}\\{\ell\in I^{(q)}}\\(\ell,i)\text{ neat}}}\left(\frac{b_{\ell i}}{\prod\limits_{\substack{i'\in I^{(q)}\\i'<\ell}}b_{i'+1,i'}}\VR_L\right).
\]
\begin{Prop}\label{twoinclusions} The following inclusion relations hold:
\begin{equation}
\left(\bigcap\limits_{q=1}^w\left(I_{\beta q}:I_{\alpha q}\right)\right)^w\subset{\rm{Ann}}_{\VR_L}\Omega_{\VR_L/\VR_K}\subset
\left(I_{\beta w}:I_{\alpha w}\right)+\left(T_w:I_{\alpha w}\right).\label{eq:twoinclusions}
\end{equation}
\end{Prop}
\begin{proof} Take an element $a\in\bigcap\limits_{q=1}^w\left(I_{\beta q}:I_{\alpha q}\right)$. For every $q\in\{1,\dots,w\}$  and every $i\in I^{(q)}$ there exists $\ell\in I^{(q+1)}$ such that the inclusion
\begin{equation}
adX_i\in\VR_LdX_{\ell}+\Omega_{\VR_L/\VR_K}(q-1)\label{eq:aXiinaXlq+1}
\end{equation}
holds. Hence the inclusion
\begin{equation}
a^qdX_i\in a^{q-1}\VR_LdX_{\ell}+a^{q-1}\Omega_{\VR_L/\VR_K}(q-1)\label{eq:aXiinaXlq+2}
\end{equation}
holds in $a^{q-1}\Omega_{\VR_L/\VR_K}$. Since \eqref{eq:aXiinaXlq+2} is valid, in particular, for $q=w$ with
\[
dX_{\ell_{w+1}}=0,
\]
we obtain, by descending induction on $q\in\{1,\dots,w\}$, that
\begin{equation}
\left(a^q\Omega_{\VR_L/\VR_K}\right)(q)\subset\left(a^{q-1}\Omega_{\VR_L/\VR_K}\right)(q-1).\label{eq:inductiononq}
\end{equation}
Since for $q=1$ the right hand side of \eqref{eq:inductiononq} is equal to $(0)$, this implies that
\[
a^w\in\text{Ann}_{\VR_L}\Omega_{\VR_L/\VR_K}.
\]
This completes the proof of the first inclusion in \eqref{eq:twoinclusions}.
\medskip

Next, take an element
\begin{equation}
a\in\text{Ann}_{\VR_L}\Omega_{\VR_L/\VR_K}.\label{ainAnnihilator}
\end{equation}
Write the module $(0)=a\Omega_{\VR_L/\VR_K}$ as $a\Omega_{\VR_L/\VR_K}=\frac{\bigoplus\limits_{i\in I^*}a\VR_LdX_i}{\mathcal J\bigcap\bigoplus\limits_{i\in I^*}a\VR_LdX_i}$. Since $adX_\ell=0$ in $\Omega_{\VR_L/\VR_K}$, $\ell\in I^{(w)}$ we have
\begin{equation}
adX_\ell\in\mathcal J\bigcap\bigoplus\limits_{i\in I^*}a\VR_LdX_i\label{eq:adXi=H}
\end{equation}
in $\bigoplus\limits_{i\in I^*}\VR_LdX_i$. On the other hand,
\begin{equation}
adX_\ell\notin\bigoplus\limits_{i\in I_{<\ell_w}}\VR_LadX_i\quad\text{in the free module }\bigoplus\limits_{i\in I^*}\VR_LadX_i.\label{eq:adXinotin<lw}
\end{equation}
According to our description of $\mathcal J$ from the previous subsections, there are three types of generators of $\mathcal J$ not belonging to
$\bigoplus\limits_{i\in I_{<\ell_w}}\VR_LadX_i$: the generators of $\bf H$ of the form $b_{i+1,i}dX_{i+1}-dQ_{i+1,i}$ with $i,i+1\in I^{(w)}$, generators of $\mathcal J_2$ and generators of the form
\[
b_{\ell(i)i}dX_{\ell(i)}-Q_{\ell(i)i},\ i\in I^{(w-1)}.
\]
Hence in the module $\frac{\bigoplus\limits_{i\in I^*}\VR_LdX_i}{\bf H}$ all the elements of
\[
\mathcal J\bigcap\left(\sum\limits_{i\in I^*}a\VR_LdX_i\setminus\sum\limits_{i\in I_{<\ell_w}}\VR_LadX_i\right)
\]
(in particular, all the elements of the form \eqref{eq:adXi=H}) belong to the module
\begin{equation}\label{eq:hugemodule}
\left(\mathcal J_2+\sum\limits_{\substack{i\in I^{(w-1)}\\{\ell\in I^{(w)}}\\(\ell,i)\text{
neat}}}(b_{\ell i})dX_{\ell}\right)\bigcap\left(a\sum\limits_{\ell\in
I^{(w)}}dX_{\ell}\right)+\sum\limits_{i\in I_{<\ell_w}}\VR_LdX_i.
\end{equation}
Remembering that for all $\ell\in I^{(w)}$ we have
\[
dX_{\ell_w}=\prod\limits_{\substack{i'\in I^{(w)}\\i'<\ell}}b_{i'+1,i'}dX_\ell\mod\sum\limits_{i\in I_{<\ell_w}}\VR_LdX_i,
\]
the module \eqref{eq:hugemodule} can be rewritten as
\[
\left(\mathcal J_2+\sum\limits_{\substack{i\in I^{(w-1)}\\{\ell\in I^{(w)}}\\(\ell,i)\text{ neat}}}\frac{b_{\ell i}}{\prod\limits_{\substack{i'\in
I^{(w)}\\i'<\ell}}b_{i'+1,i'}}dX_{\ell_w}\right)\bigcap\left(\sum\limits_{\ell\in I^{(w)}}\frac a{\prod\limits_{\substack{i'\in
I^{(w)}\\i'<\ell}}b_{i'+1,i'}}dX_{\ell_w}\right)+\sum\limits_{i\in I_{<\ell_w}}\VR_LdX_i.
\]
Thus, by definition of $\mathcal J_2$, for all $\ell\in I^{(w)}$ there exist $\ell'\in I^{(w)}$ and $i\in I^{(w-1)}$ such that
\begin{equation}
\begin{split}
&\frac a{\prod\limits_{\substack{i'\in I^{(w)}\\i'<\ell}}b_{i'+1,i'}}dX_{\ell_w}\in\\
&\in\VR_L\frac1{\prod\limits_{\substack{i'\in I^{(w)}\\i'<\ell'}}b_{i'+1,i'}}\left(\frac{\partial Q_{i_{\max}\ell'}}{\partial
X_{\ell'}}\right)_{\tilde{\bf Q}}dX_{\ell_{ w}}+\frac{b_{\ell'i}}{\prod\limits_{\substack{i'\in
I^{(w)}\\i'<\ell'}}b_{i'+1,i'}}dX_{\ell_w}+\sum\limits_{i'\in I_{<\ell_w}}\VR_LdX_{i'}.
\end{split}
\end{equation}
Recalling that $\sum\limits_{i\in I_{<\ell_w}}\VR_LdX_i=\Omega^*_{\VR_L/\VR_K}(w-1)$ and applying the isomorphism \eqref{eq:IalphaqM(q)M(q-1)} with $q=w$, we obtain
\[
va-v\left(\prod\limits_{\substack{i'\in I^{(w)}\\i'<\ell}}b_{i'+1,i'}\right)\ge\min\left\{\nu\left(\left(\frac{\partial Q_{i_{\max}\ell'}}{\partial X_{\ell'}}\right)_{\tilde{\bf Q}}\right),v\left(b_{\ell'i}\right)\right\}-v\left(\prod\limits_{\substack{i'\in I^{(w)}\\i'<\ell'}}b_{i'+1,i'}\right).
\]
We have proved that for each $\ell\in I^{(w)}$ there exist $\ell'\in I^{(w)}$ and $i\in I^{(w-1)}$ such that
\[
\frac a{\prod\limits_{\substack{i'\in I^{(w)}\\i'<\ell}}b_{i'+1,i'}}\in\VR_L\frac{\left(\frac{\partial Q_{i_{\max}\ell'}}{\partial X_{\ell'}}\right)_{\tilde{\bf Q}}}{\prod\limits_{\substack{i'\in
I^{(w)}\\i'<\ell'}}b_{i'+1,i'}}+\VR_L\frac{b_{\ell'i}}{\prod\limits_{\substack{i'\in I^{(w)}\\i'<\ell'}b_{i'+1,i'}}}\subset
I_{\beta q}+T_w.
\]
This shows that $a\in I_{\beta w}:I_{\alpha w}+T_w:I_{\alpha w}$, as desired.
\end{proof}

\begin{Exa} The point of this example is to show that both inclusions in the above proposition may be strict, even in the case of discrete rank one valuations (and so, in particular, in the absence of limit key polynomials).

Let $p$ be a prime number, $k$ a ground field of characteristic $p$ and $t$ and $u$ ---  independent variables. Put $K=k(t,u)$ and let $v$ be the $t$-adic valuation of $K$, so that
\[
\Gamma=\Z\ \mbox{ and }\ \VR_K=k(u)[t]_{(t)}.
\]
Take an integer $n\ge2$ not divisible by $p$. Let
\[
g=\left(x^p-u\right)^n-t^nx\ \mbox{ and }\ L=\frac{K[x]}{(g)}.
\]
Then $v$ admits a unique extension to $L$. We have
\[
w=2, I=(0,1,2)\ \mbox{ and }\ I^*=(0,1).
\]
A complete set of key polynomials is given by
\[
Q_0=x,\ Q_1=x^p-u\ \mbox{ and }\ Q_2=g=Q_1^n-t^nQ_0;
\]
precisely, $(Q_i)_{i\in I}$ is a complete set of key polynomials for the valuation $\nu$ on $K[x]$ and $(Q_i)_{i\in I^*}$ --- a complete set of key polynomials for $v$, viewed as a valuation of $L$. We have $\nu(x)=0$, $\nu(Q_1)=1$ and $\nu(g)=\infty$. Consequently, $\tilde Q_0=x$ and $\tilde Q_1=\frac{Q_1}t$. The valuation ring $\VR_L$ is given by
\[
\VR_L=\frac{\VR_K[X_0,X_1]}{(tX_1-X_0^p+u,X_1^n-X_0)}.
\]
We have
\[
\Omega_{\VR_L/\VR_K}=\frac{\VR_LdX_0\oplus\VR_LdX_1}{\left(tdX_1,n\left(\frac{\eta^p-u}{ t}\right)^{n-1}dX_1-dX_0\right)}.
\]
We compute
\begin{eqnarray}
\ann_{\VR_L}\Omega_{\VR_L/\VR_K}&=&(t)\\
I_{\alpha1}&=&I_{\alpha2}=I_{\beta2}=\VR_L\\
I_{\beta1}&=&(0)\\
T_2&=&(t)\\
I_{\beta1}:I_{\alpha1}&=&(0)\quad\text{and}\\
I_{\beta2}:I_{\alpha2}&=&\VR_L.
\end{eqnarray}
Thus the inclusions of \eqref{eq:twoinclusions} become $(0)\subsetneqq(t)\subsetneqq\VR_L$.
\end{Exa}
\medskip

\noindent{\bf Notation.} Let
\begin{equation}
\begin{split}
B&=\left\{vb_{\ell i}\ \left|\ i\in I^{(w-1)},\ell\in I^{(w)},(\ell,i)\text{ neat}\right.\right\}\quad\ \ \text{ if }\ w\ge2\\
&=\{+\infty\}\qquad\qquad\qquad\qquad\qquad\qquad\qquad\qquad\text{ if }\ w=1.\label{eq:bdedawayfrom0}
\end{split}
\end{equation}

\begin{Teo}\label{MainInvertibility1}\begin{enumerate}
\item If the invertibility hypothesis holds for $(L/K,v)$, then
\begin{equation}
\Omega_{\VR_L/\VR_K}\cong(0).\label{eq:omegaequalto0}
\end{equation}
\item Assume that $\Omega_{\VR_L/\VR_K}\cong(0)$ and the set $B$ is bounded away from 0 in $\frac{Lv}{\Delta_-^{(w)}}$. Then the invertibility hypothesis holds at level $w$.
\end{enumerate}
\end{Teo}

\begin{proof}[Proof of Theorem \ref{MainInvertibility1}] (1) Assume that the invertibility hypothesis holds. This amounts to saying that $[1,\dots,w]$ is an interval of invertibility. By Proposition \ref{upperboundonwidth},
\[
\text{width}\left(\Omega_{\VR_L/\VR_K}\right)=0
\]
and \eqref{eq:omegaequalto0} follows. Alternatively, note that by the invertibility hypothesis we have
\[
I_{\beta q}:I_{\alpha q}=\VR_L\mbox{ for all }q\in\{1,\dots,w\},
\]
so the first ideal in \eqref {eq:twoinclusions} is the unit ideal. Hence
\begin{equation}
\text{Ann}_{\VR_L}\Omega_{\VR_L/\VR_K}=\VR_L\label{eq:annihilatorunit}
\end{equation}
and, again, \eqref{eq:omegaequalto0} follows.
\medskip

(2) By \eqref{eq:omegaequalto0}, the equality \eqref{eq:annihilatorunit} holds. Since $B$ is bounded away from 0 in $\frac{Lv}{\Delta_-^{(w)}}$, we have
$T_w:I_{\alpha w}\subset\mathfrak m_L\subsetneqq\VR_L$. Then $I_{\beta w}:I_{\alpha w}=\VR_L$ by the second inclusion in \eqref{eq:twoinclusions}, so the invertibility hypothesis holds at level $w$, as desired.
\end{proof}


\end{document}